\documentclass{amsart}

\usepackage[utf8]{inputenc}
\usepackage[english]{babel}

\usepackage{amsmath}
\usepackage{amsfonts}
\usepackage{amsthm}
\usepackage{amssymb}
\usepackage{enumitem}
\usepackage{hyperref}
\usepackage{mathrsfs}
\usepackage{hhline}
\usepackage{ytableau}

\usepackage[doi=false,isbn=false,url=false,eprint=false,backend=bibtex,style=alphabetic,maxalphanames=98,maxbibnames=98,maxcitenames=98]{biblatex}
\DefineBibliographyStrings{english}{page={},pages={}}

\allowdisplaybreaks

\theoremstyle{plain}
\newtheorem{x}{X}[section]
\newtheorem{lem}[x]{Lemma}
\newtheorem{thm}[x]{Theorem}
\newtheorem*{ulthm}{Theorem}
\newtheorem{cor}[x]{Corollary}

\newtheorem{prop}[x]{Proposition}

\theoremstyle{definition}
\newtheorem{defn}[x]{Definition}
\newtheorem{ex}[x]{Example}

\theoremstyle{remark}
\newtheorem{rem}[x]{Remark}
\newtheorem*{ulrem}{Remark}
\newtheorem{notation}[x]{Notation}

\DeclareMathOperator{\bc}{\mathbf{c}}
\DeclareMathOperator{\CF}{CF}
\DeclareMathOperator{\Irr}{Irr}
\DeclareMathOperator{\diag}{diag}
\DeclareMathOperator{\ch}{ch}
\DeclareMathOperator{\im}{Im}
\DeclareMathOperator{\id}{id}

\DeclareMathOperator{\ev}{ev}
\DeclareMathOperator{\Ker}{Ker}

\DeclareMathOperator{\tr}{tr}

\newcommand{\Hbc}{\mathsf{H}_{\bc}}
\newcommand{\rHbc}{\overline{\mathsf{H}}_{\bc}}
\newcommand{\ple}{\hspace{-0.3em}}
\newcommand{\ssub}{\subseteq}
\newcommand{\bigslant}[2]{{\raisebox{.2em}{$#1$}\left/\raisebox{-.2em}{$#2$}\right.}}
\newcommand{\inv}{^{-1}}
\newcommand{\thh}{^{\text{th}}}

\newcommand{\bmu}{\boldsymbol{\mu}}
\newcommand{\brho}{\boldsymbol{\rho}}
\newcommand{\blambda}{\boldsymbol{\lambda}}

\newcommand{\calG}{\mathcal{G}}
\newcommand{\calS}{\mathcal{S}}

\newcommand{\frh}{\mathfrak{h}}
\newcommand{\frS}{\mathfrak{S}}

\newcommand{\bbC}{\mathbb{C}}
\newcommand{\bbN}{\mathbb{N}}
\newcommand{\bbQ}{\mathbb{Q}}
\newcommand{\bbZ}{\mathbb{Z}}

\newcommand{\scrP}{\mathscr{P}}

\title[Wreath Macdonald Polynomials at $\MakeLowercase{q}=\MakeLowercase{t}$ as characters]{Wreath Macdonald Polynomials at $q=t$ \\ as Characters of Rational Cherednik Algebras}

\usepackage[foot]{amsaddr}

\author{Dario Mathi\"a and Ulrich Thiel}
\address{Fachbereich Mathematik, Technische Universit\"at Kaiserslautern, 67663 Kaiserslautern, Germany}
\email{mathiae@mathematik.uni-kl.de}
\email{thiel@mathematik.uni-kl.de}
\date{December 20, 2021}

\begin{document}

	\maketitle
	\thispagestyle{empty}

	\begin{abstract}
		Using the theory of Macdonald \cite{macdonald2015symmetric}, Gordon \cite{gordon2003baby} showed that the graded characters of the simple modules for the restricted rational Cherednik algebra by Etingof and Ginzburg \cite{etingof2002symplectic} associated to the symmetric group $\frS_n$ are given by plethystically transformed Macdonald polynomials specialized at $q=t$. We generalize this to restricted rational Cherednik algebras of wreath product groups $C_\ell \wr \frS_n$ and prove that the corresponding characters are given by a specialization of the wreath Macdonald polynomials defined by Haiman in \cite{haiman2003combinatorics}.
	\end{abstract}

	\section*{Introduction}

	In 1882, Kostka \cite{kostka1882ueber} studied base changes between certain bases of symmetric polynomials to obtain what is known today as the \textit{Kostka number} $K_{\mu\lambda} \in \bbN$ attached to two partitions $\lambda,\mu$. In the second half of the $20\thh$ century, the Kostka numbers were generalized to a $t$-analogue $K_{\mu\lambda}(t) \in \bbN[t]$ called the \textit{Kostka--Foulkes polynomial} (see \cite{foulkes1974survey}, \cite{lascoux1978sur}). The polynomial $K_{\mu\lambda}(t)$ is a generalization in the sense that we re-obtain the Kostka numbers by specializing $K_{\mu\lambda}(t)$ at $t=1$. Macdonald \cite{macdonald1988new} then generalized this even further to a $q,t$-analogue $K_{\mu\lambda}(q,t) \in \bbN[q,t]$ called the \textit{Kostka--Macdonald coefficient} by making the far-reaching discovery of the Macdonald symmetric polynomial $P_\lambda(x;q,t)$ (see \cite[VI.4]{macdonald2015symmetric} for a precise definition). One of the main properties of the Macdonald polynomial is that it specializes to all the previous settings and more, in particular we have $K_{\mu\lambda}(0,t) = K_{\mu\lambda}(t)$. The fact that the Kostka--Macdonald coefficients are polynomials with nonnegative integer coefficients is called Macdonald positivity. It was conjectured by Macdonald and proven by Haiman \cite{haiman2001hilbert} using the transformed Macdonald polynomial $H_\lambda(x;q,t)$.

	At each stage, deep connections were made between the various specializations of Macdonald polynomials and representation-theoretic objects (see for example  \cite{nelsen2003kostka}, \cite{haiman2002notes}). The one that is of particular interest to us has been made by Gordon \cite{gordon2003baby} for $q=t$ and concerns graded characters of restricted rational Cherednik algebras. First, Etingof and Ginzburg \cite{etingof2002symplectic} have defined for a complex reflection group $W$---so, in particular for the symmetric group---the \textit{rational Cherednik algebra} $\Hbc(W)$.
	This is an infinite-dimensional $\mathbb{C}$-algebra which is a ``rational degeneration'' of the double affine Hecke algebra by Cherednik \cite{cherednik1992double}
	 that was used in proving, e.g., the Macdonald constant term conjecture \cite{cherednik1995double}. Here, $\bc$ is a parameter from a complex vector space of dimension equal to the number of conjugacy classes of reflections. We note that in \cite{etingof2002symplectic} the rational Cherednik algebra is more generally defined for an additional parameter $t \in \mathbb{C}$ (that is not to be confused with the $t$ in the Kostka--Macdonald coefficients!) and we consider the case $t=0$ here.

	The algebra $\Hbc(W)$ has a large center and admits a finite-dimensional quotient $\rHbc(W)$ called the \textit{restricted} rational Cherednik algebra. We will give a brief review of these algebras in Section \ref{sec:rrca} but their key properties are as follows: they are naturally $\mathbb{Z}$-graded and admit a triangular decomposition that leads to a theory of (graded) standard modules $M(\lambda)$ indexed by the irreducible complex characters $\lambda$ of $W$, each $M(\lambda)$ has a simple head $L(\lambda)$, and the graded shifts of the $L(\lambda)$ give all the graded simple $\rHbc(W)$-modules. Moreover, the group algebra $\mathbb{C}W$ is naturally a subalgebra of $\rHbc(W)$ concentrated in degree zero and therefore $L(\lambda)$ can be considered as a graded $W$-module. In particular, we can write
		\begin{equation}
			\lbrack L(\lambda) \rbrack = \sum_{\mu \in \Irr(W)} a_{\mu\lambda}(t) \cdot \lbrack \mu \rbrack
		\end{equation}
		in the graded Grothendieck group of $W$, for some $a_{\mu \lambda}(t) \in \mathbb{N} \lbrack t \rbrack$. We call this expression the \textit{graded character} of $L(\lambda)$. In case of the symmetric group $W=\frS_n$, where the irreducible complex characters are in bijection with partitions of $n$, Gordon has proven in \cite[Thm. 6.4$(ii)$]{gordon2003baby} that for $\bc \neq 0$ we have
	\begin{equation}
		[L(\lambda)] = \sum_{\mu} K_{\mu\lambda}(t,t) \cdot [\mu] \label{eq:gor}
	\end{equation}
	(see Section \ref{sec:rrca} for a detailed account). Using Haiman's transformed Macdonald polynomials \cite[Sec. 2.1]{haiman2001hilbert}, we can write \eqref{eq:gor} as
	\begin{equation} \label{gordon:s_n}
		\ch L(\lambda) = t^{b(\lambda)} \cdot H_\lambda(x;t,t\inv) \;,
	\end{equation}
	where $b(\lambda) = \sum_i (i-1) \lambda_i$ and $\ch L(\lambda)$ denotes image of the $W$-character of $L(\lambda)$ under the Frobenius character map, see Sections \ref{sec:A} and \ref{sec:rrca}.\\

	In the early 2000's, Haiman \cite{haiman2003combinatorics} conjectured, and Bezrukavnikov--Finkelberg \cite{bezrukavnikov2014wreath} later proved, the existence of {\it wreath Macdonald polynomials} $H_{\blambda}(x;q,t)$ associated to a multipartition $\blambda$, i.e. an $\ell$-tuple of partitions that all sum up to $n$. The background of this extension is the geometry of (resolutions of) the variety $(\mathfrak{h} \oplus \mathfrak{h}^*)/W$, where $W$ is the wreath product $C_\ell \wr \frS_n$ acting naturally on $\mathfrak{h} = \mathbb{C}^n$. Such multipartitions $\blambda$ parametrize the irreducible complex characters of $C_\ell \wr \frS_n$ (see Section \ref{sec:multi}). The group $C_\ell \wr \frS_n$ is a complex reflection group as well and so we can consider its restricted rational Cherednik algebra. In contrast to the symmetric group case the parameter space of $\bc$ is in general not one-dimensional anymore. But there is a notion of \textit{generic} parameters which can be defined as lying outside a finite set of explicitly known hyperplanes (see Section \ref{sec:generic}). Our main result (Theorem \ref{cor:Haiman}), generalizing \eqref{gordon:s_n}, is:

	\begin{ulthm}
		Let $W = C_\ell \wr \frS_n$ and let $\bc$ be a generic parameter. Then for the simple $\rHbc(W)$-module $L(\blambda)$ associated to an $\ell$-multipartition $\blambda$ of $n$, we have
		\begin{equation}
			\ch L(\blambda) = t^{b(\blambda^*)} \cdot H_{\blambda}(x;t,t\inv)\;.
		\end{equation}
	\end{ulthm}

	The specialization considered here reduces the wreath Macdonald polynomials to a plethystic transformation of the Schur functions, as it does in the traditional case, and therefore has a much simpler structure. This plethysm is used in Section \ref{sec:main} and is obtained by inverting the transformation given by the virtual character $\sum_i (-t)^i \cdot \text{char} \wedge^i \frh$ (see \cite[2.3.6]{wen2019wreath} for more details).\\

	On the Cherednik algebra side there are two natural follow-up questions. The first concerns the groups of type $G(\ell,d,n)$ for $d > 1$. These groups admit a generic case as well, although it is much more difficult to describe as the corresponding Calogero--Moser space is not generically smooth. The other one is in some sense more general than the first: it is the description of graded characters of $\rHbc(W)$ of wreath products at specialized parameter $\bc$. For small $W$ the characters can be computed using CHAMP \cite{thiel2015champ} and we give an example in Section \ref{sec:B}. At special parameter, the simple modules $L(\blambda)$ are not isomorphic to the regular representation of $W$ anymore, which makes them very difficult to model with the theory of symmetric functions. One would need a type of ``truncated'' version of the Macdonald polynomial.

\begin{ulrem}
		Griffeth \cite{griffeth2014macdonald} has made a connection between Macdonald polynomials and characters of the associated graded of irreducible modules for the restricted rational Cherednik algebra with respect to a filtration by total degree. At this time though, it remains unclear how this connects to our setting, i.e. before passing to the associated graded. Moreover, the extension to wreath products remains open in \cite[2.20]{griffeth2014macdonald}.
\end{ulrem}

	\subsection*{Acknowledgements} The first author would like to thank Luc Lapointe for very helpful comments early on in the process, and Joshua Wen for giving critical insight into Haiman's wreath Macdonald polynomials. We also thank Stephen Griffeth for explaining some of his results and Gwyn Bellamy for helpful corrections of an earlier version.

	This work is a contribution to the SFB-TRR 195 'Symbolic Tools in Mathematics and their Application' of the German Research Foundation (DFG).

	\setcounter{tocdepth}{1}
	\tableofcontents

	\section{Review of Macdonald polynomials}

	\label{sec:A}

	We begin by reviewing the basic combinatorial concepts, especially the ring of symmetric functions, the plethystic calculus, and the Macdonald polynomials. We do this in order to explain how to derive \eqref{gordon:s_n} from \eqref{eq:gor} and to prepare our generalization of the combinatorics to wreath products.

	\subsection{Partitions and Young diagrams}

	\label{sec:partitions}

	Fix once and for all some nonnegative integer $n$. We denote by $\scrP(n)$ the set of its partitions which we will identify with their corresponding Young diagrams (as left-justified rows). The partition $\lambda=(\lambda_1 \geq ... \geq \lambda_k > 0)$ has length $k$, denoted by $l(\lambda)$. The \textit{size} of $\lambda$ is the number of its boxes denoted $|\lambda|$ and the empty partition is written as $\emptyset \in \scrP(0)$. We have the \textit{dominance order} $\unlhd$ on $\scrP(n)$ defined by
		\begin{equation}
			\lambda \unlhd \mu \  :\Longleftrightarrow \ \sum_{i=1}^k \lambda_i \leq \sum_{i=1}^k \mu_i \ \text{ for all } k \geq 1\;, \label{eq:semiorder}
		\end{equation}
		where we extend the partition sequences by $0$'s when necessary. The $b$-invariant of a partition $\lambda \in \scrP(n)$ is given by
		\begin{equation}
			b(\lambda) = \sum_{i=1}^{l(\lambda)} (i-1)\cdot \lambda_i\;. \label{eq:binv}
		\end{equation}
	We will extend $b(\cdot)$ to arbitrary vectors $\alpha \in \bbN^k$ for some $k$. Also note that $b(\cdot)$ is often denoted $n(\cdot)$ by many authors, for example \cite{macdonald2015symmetric}.

	\subsection{Symmetric functions and plethysms}

	In this section we want to define Kostka--Macdonald coefficients. We refer to \cite{macdonald2015symmetric} for the general theory of Macdonald polynomials and to \cite{loehr2011computational} for the theoretical underpinnings of plethysms and symmetric functions. 	The latter additionally serves as a more modern treatment of these concepts.

	Let $x_1, ..., x_N$ be indeterminates over $\bbC$ for some fixed integer $N \geq n$. We let the symmetric group $\frS_N$ act on these variables by permuting their indices and extend this action to the polynomial ring $\bbC[x_1, ..., x_N]$. The \textit{ring of symmetric polynomials in $N$ variables} $\Lambda_N$ is the fixed point set of the action of $\frS_N$ on $\bbC[x_1, ..., x_N]$ and we write
		\begin{equation}
			 \Lambda_N = \bbC[x_1, ..., x_N]^{\frS_N}\;. \label{eq:symmpols}
		\end{equation}
		For a nonnegative integer $k$ we denote by $\Lambda_N^k$ the homogeneous symmetric polynomials of degree $k$.

	For $r \in \{1, ..., N\}$ we define the \textit{power-sum symmetric polynomial} by
	\begin{equation}
		p_{r,N} = x_1^r + ... + x_N^r
	\end{equation}
	and their product
	\begin{equation}
		p_{\lambda,N}= \prod_{i=1}^{l(\lambda)} p_{\lambda_i,N}
	\end{equation}
	for a partition $\lambda$ of $n$. A standard fact about symmetric polynomials is that $\Lambda_N$ is isomorphic to a polynomial ring, in particular
		\begin{equation}
			\Lambda_N = \bbC[p_{1,N}, ..., p_{N,N}]\;. \label{eq:LambdaNpoly}
		\end{equation}
	Furthermore, for $k \leq N$ we have
		\begin{equation}
			 \Lambda_N^k = \left<p_{\lambda,N} \mid \lambda \in \scrP(k)\right>_{\bbC} \label{eq:LambdaNk}
		\end{equation}
	as $\bbC$-vector spaces. See \cite[II.2]{macdonald2015symmetric} for an in-depth discussion and proofs.

	The Macdonald polynomials and Kostka--Macdonald coefficients are defined over the rational function field $\bbC(q,t)$ where $q$ and $t$ are indeterminates over $\bbC$. This means we will need to extend scalars in our previous definitions. Let
		\begin{equation}
			\Lambda_N (q,t) := \Lambda_N \otimes_{\bbC} \bbC(q,t)
		\end{equation}
		be the \textit{ring of symmetric polynomials in $N$ variables with parameters $q,t$}. The proofs of the identities \eqref{eq:LambdaNpoly} and \eqref{eq:LambdaNk} work the same.

		We will now define the ring of abstract symmetric functions in order to simplify the construction of the plethysm (see \cite[Sec. 2]{loehr2011computational} for details and proofs). Let $p_r$ for $r \in \bbZ_+$ be indeterminates over $\bbC$. The \textit{ring of (abstract) symmetric functions} is simply defined as the polynomial ring
			\[\Lambda=\bbC[p_1, p_2, ... ]\;.\]
		We set $\deg(p_r)=r$ and refer to the  $p_r$ as the \textit{(abstract) power-sum symmetric function of degree $r$} for $r \in \bbZ_+$. We again denote its homogeneous degree $k$ piece by $\Lambda^k$. By {\cite[Sec. 2.1]{loehr2011computational} there exists for a fixed $N \in \bbN$ an evaluation homomorphism
		\begin{equation}
		\begin{array}{rcl}
 			\ev_N : \Lambda & \to & \Lambda_N \\
 			p_r & \mapsto & p_{r,N}
 		\end{array}
		\end{equation}
 	that restricts to a vector space isomorphism $\Lambda^k \overset{\sim}{\longrightarrow} \Lambda_N^k$ for all $k \leq N$. For $f \in \Lambda$ we write $f(x_1, ..., x_N)$ for $\ev_N(f)$. Because of their isomorphic properties, we will often use $\ev_N$ implicitly and identify a symmetric function with its polynomial image.

 	We denote by
		\begin{equation}
			 \cdot   [   \cdot   ] : \Lambda \times \Lambda \to \Lambda
		\end{equation}
		the \textit{plethysm} or \textit{plethystic substitution} which is uniquely determined by the following properties:
		\begin{enumerate}
			\item For all $r_1, r_2 \in \bbZ_+$ we have $p_{r_1} [ p_{r_2}] = p_{r_2} [ p_{r_1}] = p_{r_1\cdot r_2}$\;,
			\item for all $m \geq 1$, the map \[ L_m : \Lambda \to \Lambda, \ g \mapsto p_m [g]\] is a homomorphism of $\bbC$-algebras,
			\item for all $g \in \Lambda$, the map \[ R_g : \Lambda \to \Lambda, \ f \mapsto f [g]\] is a homomorphism of $\bbC$-algebras.
		\end{enumerate}
	Existence and uniqueness of the plethysm are proven in \cite[Thm. 1]{loehr2011computational}. We record some of its additional properties in the next Lemma.
	\begin{lem}[{\cite[2.3]{loehr2011computational}}]\label{lem:ple}
		For $f,g,h \in \Lambda$, $r \geq 1$ we have
		\begin{enumerate}
			\item $f[g[h]] = (f[g])[h]$,
			\item $p_1[f] = f[p_1] = f$
			\item $(p_r [g])(x_1, ..., x_N) = g(x_1^r, ..., x_N^r)$,
		\end{enumerate}
	\end{lem}

	Following \cite[Ex. 1]{loehr2011computational}, we extend the plethystic operation to $\Lambda(q,t)$ by demanding that the indeterminates $q,t$ are treated the same as the $x_i$, i.e. we have
	\begin{equation}
		p_r[t] = t^r\;, \ p_r [q] = q^r\;.
	\end{equation}
	Note that this means that the plethysm is no longer a morphism of $\bbC(q,t)$-algebras but only of $\bbC$-algebras. This is illustrated by the following example.

	\begin{ex}
		Let $c \in \bbC$, $r,s \in \bbZ_+$. We can calculate in $\Lambda(q,t)$
		\begin{equation}
			 p_r \left[c\cdot t^2\cdot p_1 - \frac{1}{1-q} \cdot p_s\right] = c \cdot t^{2r} \cdot  p_r - \frac{1}{1-q^r} \cdot p_{sr}\;.
		\end{equation}
		When using the evaluation homomorphism we have
		\begin{equation}
			p_r [c \cdot t \cdot p_1](x_1, ..., x_N) = c \cdot t^r \cdot p_r (x_1, ..., x_N)= c \cdot t^r (x_1^r + ... + x_N^r)
		\end{equation}
		in $\Lambda_N(q,t)$.
	\end{ex}

	Since $0,1 \in \Lambda$ are symmetric functions as well, the algebra homomorphism property of the plethysm tells us that $p_r[0]=0$ and $p_r[1]=1$.

	\begin{notation}
		It is common to write the evaluation homomorphism $\ev_N$ as a plethystic substitution, namely as $f[X]$ for $f \in \Lambda$ where we define
	\begin{equation}
		 X := x_1 + ... + x_N = p_{1,N}\;.
	\end{equation}
	\end{notation}

	\subsection{Macdonald polynomials}

	We are going to reiterate the notations and results found in \cite{macdonald2015symmetric}. For a partition $\lambda$, let $s_\lambda \in \Lambda$ be its Schur function. We have
		\begin{equation}
			s_\lambda = \sum_{\mu \in \scrP(n)} z_\mu\inv \cdot \chi_\mu^\lambda \cdot p_\mu \label{eq:s<->p}
		\end{equation}
		where $\chi_\mu^\lambda$ is the evaluation of the character corresponding to $\lambda$ at cycle type $\mu$, and the number $z_\mu$ is equal to $\prod_{i\in \bbN} i^{a_i}\cdot a_i!$ with the number $i$ appearing $a_i$ times in $\mu$. For partitions $\lambda,\mu,\nu$ we define the \textit{Littlewood--Richardson coefficient} $c_{\lambda,\mu}^\nu$ by
		\begin{equation}
			s_\lambda \cdot s_\mu = \sum_{\nu} c_{\lambda,\mu}^\nu s_\nu
		\end{equation}
		and endow $\Lambda$ with a scalar product $\left<\cdot,\cdot\right>$ by demanding that the Schur functions form an orthonormal basis.

		Let $\CF(\frS_n)$ denote the space of complex-valued class functions on $\frS_n$ and let $\chi \in \CF(\frS_n)$. Equation \eqref{eq:s<->p} can be extended linearly to the \textit{Frobenius character map}
		\begin{equation}
			\begin{array}{rcl}
				\ch : \CF(\frS_n) &\to &\Lambda \\
				\chi & \mapsto & \sum_{\mu \in \scrP(n)} z_\mu\inv \cdot \chi_\mu \cdot p_\mu\;. \label{eq:Frob}
			\end{array}
		\end{equation}
		which is an isometry of vector spaces when taking the usual inner product of characters (see for example \cite[Thm 7.3]{fulton1997young} for a proof). Note that we identify an $\frS_n$-module with its character when applying $\ch$. The construction in \eqref{eq:Frob} can be lifted to $\bbC(q,t)$ as well. Lastly, one important identity of the Schur function is
		\begin{equation}
			s_\lambda \ple \left[ \frac{1}{1-t}\right] = t^{b(\lambda)} \cdot H_\lambda\inv(t) \label{eq:shook}
		\end{equation}
		with
		\begin{equation}
			H_\lambda(t) = \prod_{(i,j)\in \lambda} (1-t^{h(i,j)}) \label{eq:hook}
		\end{equation}
		where $h(i,j)$ is 1 plus the of boxes to the right and below the box with (matrix) coordinates $(i,j)$ in the Young diagram of $\lambda$. The polynomial $H_\lambda(t)$ is called the \textit{hook polynomial} of $\lambda$ (see \cite[I.3 Ex.2]{macdonald2015symmetric}.\\

		Denote by $P_\lambda(x;q,t)$ the \textit{Macdonald symmetric function} as defined in \cite[VI.4]{macdonald2015symmetric} and by $J_\lambda(x;q,t)$ its integral version           \cite[VI.8]{macdonald2015symmetric}. Furthermore denote by $H_\lambda(x;q,t)$ the transformed version of Haiman given for example in \cite[Sec. 2.1]{haiman2001hilbert}. We have the relationship
		\begin{equation}
			H_\lambda(x;q,t) = t^{b(\lambda)}\cdot J_\lambda(x;q,t\inv)\left[\frac{X}{1-t\inv}\right]\;.
		\end{equation}
		The \textit{$q,t$-Kostka--Macdonald coefficients} are given by
		\begin{equation}
			t^{b(\lambda)}\cdot H_\lambda(x;q,t\inv) = \sum_{\mu \in \scrP(n)} K_{\mu\lambda}(q,t) \cdot s_\mu\;. \label{eq:kostkamac}
		\end{equation}
	The tables of $K_{\mu\lambda}(q,t)$ for $n \leq 6$ can be found in \cite[VI.8]{macdonald2015symmetric}.

	The key aspect of Macdonald polynomials is that they specialize to a wide range of families of symmetric functions. One such specialization is given by setting $q=t$ where we have
	\begin{align*}
		P_\lambda(x;t,t)&=s_\lambda\;, \vphantom{\left[\frac{X}{1-t\inv}\right]}\\
		J_\lambda(x;t,t)&=H_\lambda(t) \cdot s_\lambda\;.
	\end{align*}
	We will mainly be working with
	\begin{equation}
		G_\lambda(x;t,t):= t^{b(\lambda)} \cdot H_\lambda(x;t,t\inv)=H_\lambda(t) \cdot s_\lambda\left[\frac{X}{1-t}\right]\;. \label{eq:G}
	\end{equation}
	The polynomial $G_\lambda(x;t,t)$ will turn out to be the character of an irreducible module of the restricted rational Cherednik algebra which we will define next.

	\section{Restricted rational Cherednik algebras} \label{sec:rrca}

	We give a brief summary of the representation theory of restricted rational Cherednik algebras. For a more detailed discussion we refer the reader to the papers by Etingof--Ginzburg \cite{etingof2002symplectic} and  Gordon \cite{gordon2003baby}, and the survey \cite{thiel2017restricted}.

	\subsection{Complex reflection groups}

	Let $W$ be a complex reflection group acting on a complex reflection representation $\frh$. We denote by $\bbC[\frh]$ the symmetric algebra of $\frh^*$ and by $\bbC[\frh]^W$ the ring of invariants of $W$. Denote by $\bbC[\frh]^W_+$ the set of invariants with no constant term and define the \textit{ring of coinvariants} of $W$ by
	\begin{equation}
		\bbC[\frh]_W = \bigslant{\bbC[\frh]}{\bbC[\frh] \cdot \bbC[\frh]^W_+}\;. \label{eq:coinvslant}
	\end{equation}
	The action of $W$ on  $\bbC[\frh]$ descends to the coinvariant algebra and gives $\bbC[\frh]_W$ the structure of a $W$-module. From the theory of complex reflection groups we know that
		\begin{equation}
			\bbC[\frh]_W \cong \bbC W \label{eq:regrep}
		\end{equation}
	as $W$-modules, where $\bbC W$ is the group ring affording the regular representation of $W$ (see \cite{chriss1997representation}, \cite{steinberg1975theorem} for a discussion and proof). This means that for any irreducible representation $\lambda$ of $W$ its multiplicity inside $\bbC[\frh]_W$ is equal to $\dim(\lambda)$.

	Since the ideal in \eqref{eq:coinvslant} by which we take the quotient has a homogeneous generating set, we obtain a grading of $\bbC[\frh]_W$ as well. Let $G$ be any group. When we talk about \textit{graded} modules of $G$, we mean a $\bbZ$-graded vector space $V$ with a $G$-action that preserves the homogeneous degree pieces of $V$. Equivalently, we can define the group ring $\bbC G$ as a graded ring concentrated in degree 0. The graded irreducibles modules of $G$ are then given by all $\bbZ$-grade shifts of the ungraded irreducibles $G$-modules (which we view as concentrated in degree 0). The \textit{$i$-shift} $U[i]$ of a $\bbZ$-graded vector space $U$ is defined to be the graded vector space for which
	\begin{equation}
		U[i]_j = U_{j-i}
	\end{equation}
	holds. Now we have that all irreducible graded modules of a group $G$ are given by the set
	\begin{equation}
		\{ \lambda[i] \mid \lambda \in \Irr(G), i \in \bbZ\}\;.
	\end{equation}
	Let $\lambda$ be an ungraded complex simple $G$-module and $V$ be any ungraded complex $G$-module. We denote by
	\begin{equation}
		[V:\lambda] \in \bbN
	\end{equation}
	the multiplicity of the character of $\lambda$ inside the character of $V$. If $V$ is a graded $G$-module with homogeneous degree $i$ piece $V_i$ for $i \in \bbZ$, we define the \textit{graded} multiplicity of $\lambda$ in $V$ by
	\begin{equation}
		[V:\lambda]^{\text{gr}} = \sum_{i \in \bbZ} [V_i:\lambda] \cdot t^i \in \bbN[t,t\inv]
	\end{equation}
	for a complex indeterminate $t$.

	Since $\bbC[\frh]_W$ is a graded version of $\bbC W$, it is natural to ask the question in which degree pieces the copies of $\lambda$ appear inside $\bbC[\frh]_W$. This information is captured in the fake degree of $\lambda$.

	\begin{defn}\label{defn:fakedeg}
		Let $\lambda$ be an irreducible representation of a complex reflection group $W$. We define its \textit{fake degree} by
		\begin{equation}
			f_\lambda(t) = [\bbC[\frh]_W : \lambda]^{\text{gr}} \;.
		\end{equation}
		We furthermore define
		\begin{equation}
			\bar f_\lambda(t) := t^{-b}\cdot f_\lambda(t)
		\end{equation}
		where $b$ is the smallest degree appearing with nonzero coefficient in $f_\lambda(t)$. The number $b$ is called \textit{trailing degree} of (the fake degree of) $\lambda$.
	\end{defn}

	We will mostly be interested in wreath products $C_\ell \wr \frS_n$ of a symmetric group $\frS_n$ with a cyclic group $C_\ell$. These groups are complex reflection groups with respect to their natural action on $\mathbb{C}^n$ and in this context they are often denoted by $G(\ell,1,n)$. Note that for $\ell=1$ we simply obtain the symmetric group $\frS_n$.

	\subsection{Restricted rational Cherednik algebras}

	\label{subsec:RRCA}

	Let $(W,\frh)$ be a complex reflection group and let $\calS \ssub W$ be its set of complex (pseudo-)reflections. Let $\frh^*$ denote the dual module of $\frh$. We then define the $W$-module $V:=\frh \oplus \frh^*$ with tensor algebra $T(V)$. The \textit{skew tensor algebra} $T(V) \rtimes W$ is the complex vector space $T(V) \otimes \bbC W$ with multiplication given by
		\begin{equation}
			w \cdot v = w.v \cdot w
		\end{equation}
		for $w \in W$, $v \in T(V)$. The vector space $V$ becomes symplectic by defining the symplectic form
	\begin{equation}
		\omega: V \times V \to \bbC, ((y,x),(y',x')) \mapsto y(x') - y'(x)\;.
	\end{equation}
	For any $s \in \calS$ we define $\omega_s$ to be equal to $\omega$ on the 1-dimensional space $\im(s-\id_V)$ and 0 on $\Ker(s-\id_V)$. Let $\bc : \calS \to \bbC$ be a function with $\bc(s)=\bc(s')$ whenever $s$ and $s'$ are conjugate in $W$. We call $\bc$ equivariant with respect to the conjugacy action of $W$ on $\calS$.

	\begin{defn}[\cite{etingof2002symplectic}]
		The \textit{rational Cherednik algebra for $W$ at parameter $\bc$ and ``at $t=0$''} is defined as
		\begin{equation}
			\Hbc := \  \bigslant{T(V) \rtimes W}{\left\langle[x,y] - \sum_{s\in \calS} \bc(s) \omega_s(x,y)s \mid x \in \mathfrak{h}, y \in \mathfrak{h^*}\right\rangle}\;. \label{eq:Hbc}
		\end{equation}
	\end{defn}

	\begin{rem}
		There is a more general version of this algebra that includes an additional parameter $t$ which in our case is equal to 0. To differentiate these two cases, one writes \mbox{``at $t=0$''} when talking about this specific class of Cherednik algebras. That $t$ is \textit{not to be confused} with the parameter $t$ of our symmetric functions and graded modules.
	\end{rem}

	The rational Cherednik algebra admits a $\bbZ$-grading given by $\deg(\frh^*) = 1, \deg(\frh) = -1$, and $\deg(W) = 0$. We also have a triangular decomposition and so-called PBW property by {\cite[Thm 1.3]{etingof2002symplectic}}, i.e. $\Hbc$ has a graded vector space decomposition
		\begin{equation}
			\Hbc \cong \bbC[\frh] \otimes \bbC W \otimes \bbC[\frh^*]\;.
		\end{equation}

	From the definition of $\omega$ we can tell that $\omega|_{\mathfrak{h}} \equiv \omega|_{\mathfrak{h}^*} \equiv 0$. This holds in particular for all $\omega_s$. Thus the commutator relations vanish and we have
		\begin{equation}
			[x,x']=[y,y']=0
		\end{equation}
		for all $ x,x' \in \frh^*, y,y' \in \frh$. By \cite[Prop. 4.15]{etingof2002symplectic}, we can extend this property multiplicatively to $\bbC[\frh]$ and $\bbC[\frh^*]$ such that we have
		\begin{equation}
			\bbC[\frh]^W \otimes \bbC[\frh^*]^W \ssub Z(\Hbc) \label{eq:rrcaideal}
		\end{equation}
		where $Z(\Hbc)$ denotes the center of $\Hbc$. Using this, one can define a certain finite dimensional quotient of $\Hbc$, called the \textit{restricted rational Cherednik algebra} by
		\begin{equation}
			\rHbc := \bigslant{\Hbc}{A_+ \Hbc}
		\end{equation}
	where $A_+ \ssub \bbC[\frh]^W \otimes \bbC[\frh^*]^W$ is the set of elements with no constant term. The PBW property of $\Hbc$ in {\cite[Thm 1.3]{etingof2002symplectic}} descends to $\rHbc$ such that we get a graded vector space decomposition
		\begin{equation}
			\rHbc \cong \bbC[\frh]_W \otimes \bbC W \otimes \bbC[\frh^*]_W \label{eq:rrcapbw}
		\end{equation}
	where $\bbC[\frh]_W$ is the coinvariant space of \eqref{eq:coinvslant}. Using \eqref{eq:regrep} we get
		\begin{equation}
			\dim(\rHbc) = |W|^3\;.
		\end{equation}
	In particular, the restricted rational Cherednik algebra is finite dimensional.

	Using the methods of \cite{holmes1991brauer} on \eqref{eq:rrcapbw}, Gordon defined in \cite{gordon2003baby} a subalgebra of $\rHbc$ which is generated by the negative degree coinvariants and the elements of $W$ inside $\rHbc$. This gives us
		\begin{equation}
			\mathsf{B}_c := \bbC[\frh^*] \rtimes W \ssub \rHbc\;.
		\end{equation}
		For a $W$-module $\lambda$, we define an action of $q \otimes w \in \mathsf{B}_c$ on $\lambda$ by
		\begin{equation}
			(q \otimes w).v := q(0)\cdot w.v
		\end{equation}
	for all $v \in \lambda$, which turns $\lambda$ into a $\mathsf{B}_{\bc}$-module. Let $\Irr(W)$ be a set of representatives of the isomorphism classes of simple $W$-modules. We identify $\lambda \in \Irr(W)$ with the graded version concentrated in degree 0. We now induce $\lambda \in \Irr(W)$ to
		\begin{equation}
			M(\lambda) := \rHbc \otimes_{\mathsf{B}_{\bc}} \lambda \label{eq:Mdef}
		\end{equation}
	called the \textit{standard} (or \textit{baby Verma}) \textit{module of $\lambda$}. We have
	\begin{equation}
		M(\lambda) = \bbC[\frh]_W \otimes \lambda \label{eq:std}
	\end{equation}
	by the vector space decomposition of $\rHbc$ \eqref{eq:rrcapbw}.

	For any $W$-representation $V$ denote by $[V]$ its isomorphism class in the graded Grothendieck group $\calG_{\text{gr}}(W)$ of $W$. When viewing equation \eqref{eq:std} in $\calG_{\text{gr}}(W)$ we can use the fake degree from Definition \ref{defn:fakedeg} to obtain
	\begin{equation}
			[M(\lambda)] = \sum_{\mu \in \Irr(W)} f_\mu(t) \cdot [\mu \otimes \lambda]\;. \label{eq:Mchar}
	\end{equation}

	From this simple construction of the standard modules we are actually able to give a complete set of pairwise nonisomorphic simple graded modules of $\rHbc$.

	\begin{thm}[{\cite[Prop. 4.3]{gordon2003baby}}]
		Each $\rHbc$-module $M(\lambda)$ has a simple head, denoted by $L(\lambda)$. Furthermore, the set
		\begin{equation}
			\{L(\lambda)[i] \mid \lambda \in \Irr(W), i \in \bbZ \}
		\end{equation}
		is a complete set of pairwise nonisomorphic simple graded $\rHbc$-modules.
	\end{thm}

	Note that we have $L(\lambda[i]) \cong L(\lambda)[i]$ in the construction above. We will focus on the description of $L(\lambda)$ with $\lambda$ concentrated in degree 0. Alternatively, we can characterize $L(\lambda)$ by being concentrated in nonnegative degree and having a copy of $\lambda$ in its degree 0 piece.

	When studying the representation theory of restricted rational Cherednik algebra, we will mainly be interested in the relationship of three classes of modules: standards, simples, and simples of $W$, which are respectively denoted by
	\begin{equation}
		(M(\lambda))_{\lambda\in \Irr(W)}\;, \ \ \ (L(\lambda))_{\lambda\in \Irr(W)}\;, \ \ \ (\lambda)_{\lambda\in \Irr(W)}\;.
	\end{equation}
	When working in the graded Grothendieck group of $W$ we can define transition matrices over $\bbN[t,t\inv]$ between these classes of modules in the respective graded Grothendieck groups.

	\begin{defn}[{\cite[3.7]{bellamy2018highest}}] \label{defn:decomp}
	We define the matrices
	\begin{equation}
		D_\Delta, C_\Delta, C_L \in \bbZ[t,t\inv]^{\Irr(W) \times \Irr(W)}
	\end{equation}
	by
		\begin{align}
			D_\Delta &= \left([M(\lambda):L(\mu)]^{\text{gr}}\right)_{\lambda,\mu \in \Irr(W)} \;,\\
			C_\Delta &= \left([M(\lambda):\mu]^{\text{gr}}\right)_{\lambda,\mu \in \Irr(W)} \;,\\
			C_L &= \left([L(\lambda):\mu]^{\text{gr}}\right)_{\lambda,\mu \in \Irr(W)} \;.
		\end{align}

	\end{defn}
	The three matrices are related by
		\begin{equation}
			C_\Delta = D_\Delta \cdot C_L\;. \label{eq:CD=C}
		\end{equation}
	We are mainly interested in the matrix $C_L$, i.e. the decomposition of simple modules of the restricted rational Cherednik algebra into simple modules of $W$. Because of \eqref{eq:Mdef} we have that $C_\Delta$ specializes to the identity matrix for $t=0$ and is therefore invertible over $\bbQ(t)$. This makes $D_\Delta$ and $C_L$ invertible over $\bbQ(t)$ as well and we can transform \eqref{eq:CD=C} into
		\begin{equation}
			 C_L = D_\Delta\inv \cdot C_\Delta\;. \label{eq:C=DC}
		\end{equation}
	Since the matrix $C_\Delta$ is completely controlled by the fake degrees, which are easy to compute (see Theorem \ref{thm:fakedeg}), the matrices $C_L$ and $D_\Delta$ are in some sense equally difficult to attain.

	\subsection{Generic parameters} \label{sec:generic}

	Even though we dropped this from the notation, the representation theory of $\rHbc$, and thus the decomposition matrices in Definition \ref{defn:decomp}, depend on the parameter $\bc$. In \cite[Sec. 3]{thiel2017restricted} a notion of \textit{generic} parameters was introduced which relies on the general theory in \cite{Thiel-Dec, Thiel-Flat}. Intuitively, the representation theory of $\rHbc$ is ``the same'' for all generic $\bc$. To make this more precise, let $\mathbf{C} := (\mathbf{C}_s)_{s \in \mathcal{S}}$ be a set of indeterminates over $\mathbb{C}$ such that $\mathbf{C}_s = \mathbf{C}_t$ whenever $s$ and $t$ are conjugate. Then one can define the restricted rational Cherednik algebra also over the rational function field $\mathbb{C}(\mathbf{C})$. Let us denote this (generic) algebra by $\overline{\mathsf{H}}$. It follows from \cite{Thiel-Flat} that the blocks of $\rHbc$ are unions of blocks of $\overline{\mathsf{H}}$. We thus call $\bc$ {\it block-generic} if the blocks coincide. On the other hand, it follows from \cite{Thiel-Dec} that there is a map, called {\it decomposition map}, from the (graded) Grothendieck group of $\overline{\mathsf{H}}$ to the one of~$\rHbc$ which is given by choosing and reducing lattices of simple $\overline{\mathsf{H}}$-modules over a (localization of) the polynomial ring $\mathbb{C} \lbrack \mathbf{C} \rbrack$. A parameter $\bc$ is called {\it decomposition-generic} if this map is a permutation. For decomposition-generic parameters, the matrices in Definition \ref{defn:decomp} do not depend on the particular parameter by construction of the decomposition map. It is shown in \cite{Thiel-Dec,Thiel-Flat} that both types of generic parameters form a non-empty Zariski open subset of the parmater space and that block-generic parameters are decomposition-generic. Moreover, it follows from \cite{bellamy2018hyperplane} that the set of block-generic parameters is the complement of a finite hyperplane arrangement.

	What is important for us is that for the complex reflection groups $G(\ell,1,n)$ the two types of generic parameters coincide by \cite[Cor 3.22]{thiel2017restricted} and the hyperplane arrangement of non-generic parameter is explicitly known, see \cite{bellamy2018hyperplane}. We furthermore have the following important property of simple modules at generic parameters.

	\begin{prop}\label{prop:|W|}
		Let $W$ be of type $G(\ell,1,n)$ and $\rHbc$ the restricted rational Cherednik algebra of $W$ with parameter $\bc$. If $\bc$ is generic, we have
		\begin{equation}
			\dim(L(\lambda)) = |W|
		\end{equation}
		and $D_\Delta$ is a diagonal matrix with entries $\bar f_\lambda(t)$ for all $\lambda \in \Irr(W)$.
	\end{prop}

	\begin{proof}
		Since the Calogero-Moser space of $G(\ell,1,n)$ is generically smooth by \cite[\-1.12-14]{etingof2002symplectic}, we can use \cite[5.2+5.5]{gordon2003baby} and the proof of \cite[5.6]{gordon2003baby}.
	\end{proof}

	The above proposition allows us to use \eqref{eq:C=DC} to compute the matrix $C_L$.

	\subsection{Gordon's character formula}

		Gordon \cite[Thm. 6.4$(ii)$]{gordon2003baby} proved the following remarkable formula about the matrix $C_L$ in case of $W=G(1,1,n)=\frS_n$. Recall that there is a natural bijection between the irreducible characters of $W$ and partitions of $n$. Furthermore, note that $W$ has just a single conjugacy class of reflections, so the space of the parameters $\bc$ is one-dimensional and any non-zero parameter is generic.

		\begin{thm}[Gordon] \label{thm:AKostka}
		Let $W=\frS_n$ and let $\bc \neq 0$. Then the decomposition of the simple $\rHbc(W)$-module $L(\lambda)$ in the graded Grothendieck group of $W$ is given by
			\begin{equation}
				[L(\lambda)] = \sum_{\mu\in\scrP(n)} K_{\mu\lambda}(t,t) \cdot [\mu]\;.
			\end{equation}
		\end{thm}

		In the language of decomposition matrices from Definition \ref{defn:decomp}, we can rewrite Theorem \ref{thm:AKostka} as
		\begin{equation}
				C_L = {(K_{\mu\lambda}(t,t))}_{\mu, \lambda \in \Irr(\frS_n)}
		\end{equation}
		with the Kostka--Macdonald coefficients as defined in \eqref{eq:kostkamac}. Using $G_\lambda(x;t,t)$ of \eqref{eq:G} we can also say
		\begin{equation} \label{gordon:s_n_2}
			\ch L(\lambda) =  G_\lambda(x;t,t)
		\end{equation}
		where $\ch L(\lambda)$ denotes image of the $W$-character of $L(\lambda)$ under the Frobenius character map \eqref{eq:Frob}.

	\section{Generalization to wreath products}

	We now come to the main part of this paper. We are going to generalize Gordon's character formula in Theorem \ref{thm:AKostka} and the implied relation \eqref{gordon:s_n_2} to restricted rational Cherednik algebras at generic parameters for wreath products $C_\ell \wr \frS_n$. Our strategy is to first generalize $t,t$-Kostka--Macdonald coefficients and the function $G_\lambda(x;t,t)$ to a multipartition setting and then use the techniques in \cite{gordon2003baby} to generalize Theorem \ref{thm:AKostka}. From there, we will derive our main result, Theorem \ref{cor:Haiman}, which relates the character formula under the Frobenius character map to Haimain's wreath Macdonald polynomials.

	\subsection{Multisymmetric functions} \label{sec:multi}
	Most of what we mention in this section can be gleaned from the more general \cite[App. B]{macdonald2015symmetric} which we cut to our needs.

	Fix some positive integers $n$, $\ell$, and let $W$ be the complex reflection group of type $G(\ell,1,n)$. An \textit{$\ell$-multipartition of n} is an $\ell$-tuple
		\begin{equation}
			\blambda = (\lambda^{(0)}, ..., \lambda^{(\ell-1)})
		\end{equation}
		of (possibly empty) partitions such that the sum $|{\blambda}|$ of their sizes  is $n$. We denote the set of all $\ell$-multipartition of $n$ by $\scrP(\ell,n)$ and write $\boldsymbol{\emptyset} \in \scrP(\ell,0)$ for the empty $\ell$-multipartition. The Young diagram of $\blambda$ is given by the $\ell$-tuple of the Young diagrams of the $\lambda^{(i)}$. As in the $\frS_n$ case, we have that the conjugacy classes of $C_\ell \wr \frS_n$ are parameterized by the set of $\ell$-multipartitions of $n$. Furthermore, we have a natural bijection
		\begin{equation}
			\Irr(C_\ell \wr \frS_n) \leftrightarrow \scrP(\ell,n)\;.
		\end{equation}
		We refer the reader to \cite{specht1933eine} and \cite{stembridge1989eigenvalues} for the construction.\\

	Fix some primitive $\ell\thh$ root of unity $\zeta_\ell$ and let
	\begin{equation}
		\{x_i^{(j)} \mid 1 \leq i \leq N, \  0\leq j \leq \ell-1\}
	\end{equation}
	be a set of indeterminates over $\bbC$ and let $\Lambda_N^{(j)}$ denote the ring of symmetric polynomials in the variables $x_i^{(j)}$ for $0\leq j \leq \ell-1$. We define the \textit{ring $\boldsymbol{\Lambda}_N$ of $\ell$-mul\-ti\-sym\-met\-ric polynomials in $N$ variables} by
	\begin{equation}
		\boldsymbol{\Lambda}_N = \bigotimes_{j=0}^{\ell-1} \Lambda_N^{(j)}
	\end{equation}
	For a nonnegative integer $k$ we denote by $\boldsymbol{\Lambda}_N^k$ the homogeneous multisymmetric polynomials of degree $k$. Define the \textit{ring of $\ell$-multisymmetric functions} as
		\begin{equation}
			\boldsymbol{\Lambda} = \Lambda^{\otimes \ell}\;.
		\end{equation}
		Its grading is defined via the sum of the degrees of the elementary tensors, i.e.
		\begin{equation}
			 \deg\left(\bigotimes_{i=0}^\ell p_{r_i}\right) = \sum_{i=0}^{\ell-1} r_i\;.
		\end{equation}
	We extend our convention and write for $0 \leq j \leq \ell-1$
	\begin{equation}
		X^{(j)} := x_1^{(j)} + ... + x_N^{(j)}\;.
	\end{equation}
	For a fixed $N$ there exists an evaluation homomorphism
		\begin{equation}
		\begin{array}{rcl}
 			\textnormal{\textbf{ev}}_N : \boldsymbol{\Lambda} & \to & \boldsymbol{\Lambda}_N \\
 			\bigotimes_{j=0}^{\ell-1} p_{r_j} & \mapsto & \ple \bigotimes_{j=0}^{\ell-1} p_{r_j} \ple \left[ \sum_{i=0}^{\ell-1} \zeta^{i\cdot j} X^{(i)} \right]
 		\end{array}
		\end{equation}
 		that restricts to a vector space isomorphism $\boldsymbol{\Lambda}^k \overset{\sim}{\longrightarrow} \boldsymbol{\Lambda}_N^k$ for all $k \leq N$. The proof is the same as in the symmetric group case of \cite[2.1]{loehr2011computational}.

	As was the case in type $A$, we also want to extend our scalars from $\bbC$ to $\bbC(q,t)$ for some complex indeterminates $q,t$. We define
		\begin{equation}
			\boldsymbol{\Lambda}(q,t) := \boldsymbol{\Lambda} \otimes_\bbC \bbC(q,t)\;, \ \
		\boldsymbol{\Lambda}_N(q,t) := \boldsymbol{\Lambda}_N \otimes_\bbC \bbC(q,t)\;.
		\end{equation}

	On our road to the multipartition version of the $t,t$-Kostka--Macdonald coefficients, we again start with the Schur functions. For $\blambda = (\lambda^{(0)}, ..., \lambda^{(\ell-1)}) \in \scrP(\ell,n)$ we define the \textit{$\ell$-multi-Schur function of $\blambda$} as
		\begin{equation}
			s_{\blambda} = \bigotimes_{j=0}^{\ell-1} s_{\lambda^{(j)}} \in \boldsymbol{\Lambda}
		\end{equation}
		and we define a scalar product $\left<\cdot ,\cdot \right>$ on $\boldsymbol{\Lambda}$ by again demanding that the Schur functions form an orthonormal basis. Using the type $A$ Frobenius character map and $\textnormal{\textbf{ev}}_N$, one can show
		\begin{equation}
			\textnormal{\textbf{ev}}_N (s_{\blambda}) = \prod_{j=0}^{\ell-1} s_{\lambda^{(j)}}\ple \left[X^{(j)}\right]
		\end{equation}
		(see \cite[Cor. 3]{poirier1998cycle} for a proof). Recall that $\CF(G)$ denotes the set of complex-valued class functions of a group $G$. We only give the polynomial version of the Frobenius character map here, the symmetric function version is obtained by expanding the plethysm and using $\mathbf{ev}_N$. For $\bmu \in \scrP(\ell,n)$, define $z_{\bmu}$ as the size of the stabilizer of the conjugacy class corresponding to $\bmu$ inside $C_\ell \wr \frS_n$Remar.

	\begin{defn}[\cite{poirier1998cycle}]\label{defn:poi}
		For a class function $\chi \in \CF(C_\ell \wr \frS_n)$ we again write $\chi_{\bmu}$ for the evaluation of $\chi$ at the class of cycle type $\bmu$. We define the \textit{Frobenius character map} of $C_\ell \wr \frS_n$ as
		\begin{equation}
			\begin{array}{rcl}
				\ch : \CF(C_\ell \wr \frS_n) &\to &\boldsymbol{\Lambda}_N \\
				\chi & \mapsto & \sum_{\bmu \in \scrP(\ell,n)} z_{\bmu}\inv \cdot \chi_{\bmu} \cdot \prod_{j=0}^{\ell-1}  p_{\mu^{(j)}}\hspace{-0.3em} \left[ \sum_{i=0}^{\ell-1} \overline{\zeta^{i\cdot j}} X^{(i)} \right]\;.
			\end{array}
		\end{equation}
	\end{defn}

	\begin{thm}[{\cite[Thm. 2]{poirier1998cycle}}]
		The Frobenius character map is an isometry of $\bbC$-vector spaces with the scalar product on $\CF(\frS_n)$ being given by the inner product of characters.
	\end{thm}

	As is the case with $\frS_n$, we identify a $C_\ell \wr \frS_n$-module with its character when applying $\ch$.

	The Littlewood--Richardson coefficients $c_{\mu\lambda}^\nu$ of type $A$ also generalize to the multipartition setting. Because for $c_{\mu\lambda}^\nu$ to be nonzero, we have to have $|\lambda| + |\mu| = |\nu|$, we can think of $(\mu,\lambda)$ as a bipartition of $|\nu|$ and generalize the Littlewood-Richardson coefficients to products of Schur polynomials with $\ell$ factors.
	\begin{defn}
		For an $\ell$-multipartitions $\brho \in \scrP(\ell,n)$ we write
		\begin{equation}
			\prod_{i=0}^{\ell-1} s_{\rho^{(i)}} = \sum_{\nu \in \scrP(n)} c_{\brho}^\nu s_{\nu} \in \Lambda\; .
		\end{equation}
	\end{defn}

	There are two very important properties of the Schur functions when it comes to plethystic substitutions. The first relates to sums of variables.

	\begin{prop}[{\cite[3.2]{loehr2011computational}}] \label{prop:sX+Y}
		For a partition $\lambda \in \scrP(n)$ and $0 \leq i \neq j \leq \ell-1$  we have
		\begin{equation}
			s_\lambda\left [X^{(i)} + X^{(j)} \right] = \sum_{(\rho^{(0)},\rho^{(1)}) \in \scrP(2,n)} c_{\rho^{(0)}\rho^{(1)}}^\lambda \cdot s_{\rho^{(0)}}\left [X^{(i)} \right] \cdot s_{\rho^{(1)}}\left [X^{(j)} \right]\;.
		\end{equation}
	\end{prop}
	\noindent We can expand Proposition \ref{prop:sX+Y} iteratively to arbitrary sums of variables.

	\begin{cor} \label{cor:sX+Y}
		For a partition $\lambda \in \scrP(n)$  we have
		\begin{equation}
			s_\lambda\left [\sum_{i=0}^{\ell-1} X^{(i)} \right] = \sum_{\brho \in \scrP(\ell,n)} c_{\brho}^\lambda \cdot \prod_{i=0}^{\ell-1} s_{\rho^{(i)}}\left [X^{(i)} \right] \;.
		\end{equation}
	\end{cor}

	The second important property of the Schur functions has to do with the fake degree.

	\begin{thm}[{\cite[Thm 5.3]{stembridge1989eigenvalues}}] \label{thm:fakedeg}
		For a partition $\lambda \in \scrP(n)$ let $f_\lambda (t)$ denote its fake degree from Definition \ref{defn:fakedeg}. We then have
		\begin{equation}
			f_\lambda (t) = (1-t)\cdot (1-t^2) \cdots (1-t^n) \cdot t^{b(\lambda)} \cdot H_\lambda\inv(t)
		\end{equation}
		which gives us using \eqref{eq:shook}
		\begin{equation}
			f_\lambda (t) = (1-t)\cdot (1-t^2) \cdots (1-t^n) \cdot s_\lambda \ple \left[\frac{1}{1-t}\right]\;.
		\end{equation}
		For a multipartition $\blambda \in \scrP(\ell,n)$ with fake degree $f_{\blambda}(t)$ we have
		\begin{equation}
			f_{\blambda} (t) = t^{b(\alpha(\blambda))} \cdot (1-t^{\ell})\cdot (1-t^{2\cdot \ell}) \cdots (1-t^{n \cdot \ell}) \cdot \prod_{i=0}^{\ell-1} t^{\ell \cdot b(\lambda^{(i)})} \cdot H_\lambda\inv(t^\ell) \label{eq:fblambda}
		\end{equation}
		where $b(\alpha(\blambda))$ is the $b$-invariant of the vector $\alpha(\blambda)=(|\lambda^{(0)}|, ..., |\lambda^{(\ell-1)}|)$ given in \eqref{eq:binv}. Expressing \eqref{eq:fblambda} in Schur functions we obtain
		\begin{equation}
			f_{\blambda} (t) = t^{b(\alpha(\blambda))} \cdot (1-t^{\ell})\cdot (1-t^{2\cdot \ell}) \cdots (1-t^{n \cdot \ell}) \cdot \prod_{i=0}^{\ell-1} s_{\lambda^{(i)}} \ple \left[\frac{1}{1-t^\ell}\right]\;.
		\end{equation}
	\end{thm}

	\begin{defn}
		We define for $\blambda \in \scrP(\ell,n)$ the number
		\begin{equation}
			b(\blambda) = b(\alpha(\blambda)) + \ell \cdot \sum_{i=0}^{\ell-1} b(\lambda^{(i)})\;.
		\end{equation}
	\end{defn}

	\begin{rem} \label{rem:barf}
		When we set $\ell=1$, the definition of $b(\blambda)$ agrees with $b(\lambda)$. Also,
		using Theorem \ref{thm:fakedeg} we can see that the lowest power appearing with nonzero coefficient in $f_{\blambda}$ is in degree $b(\blambda)$. This means for the grade-shifted version of the fake degree $\bar{f}_{\blambda}(t)$ from Definition \ref{defn:fakedeg} we have
		\begin{equation}
			\bar{f}_{\blambda}(t) = t^{-b(\blambda)} \cdot f_{\blambda}(t)\;.
		\end{equation}
	\end{rem}

	In \cite{haiman2003combinatorics} Haiman defined wreath Macdonald polynomials
	\begin{equation}
		H_{\blambda}(x;q,t)
	\end{equation}
	for $C_\ell \wr \frS_n$ which were proven to exist in \cite{bezrukavnikov2014wreath}. We will not review their definition here as that would require a lot more combinatorial overhead. For an in-depth discussion we refer to \cite{wen2019wreath}.

	\subsection{$t,t$-Kostka--Macdonald coefficients}

	We define multisymmetric functions $G_{\blambda}(x;t,t)$ which will give us the graded characters of simple modules of the restricted rational Cherednik algebra in the wreath product case. The $G_{\blambda}(x;t,t)$ will turn out to be a specialized versions of the wreath Macdonald polynomials defined by Haiman in \cite{haiman2003combinatorics} (see Theorem \ref{cor:Haiman}).

	Define for $0 \leq j \leq \ell-1$ the variables
		\begin{equation}
			Z^{(p)} = \sum_{i=0}^{\ell-1}  t^{\overline{(i-p)}} X^{(i)} \in \left<X^{(i)} \mid 0\leq i \leq \ell-1\right> _{\bbC(t)} \label{eq:Z}
		\end{equation}
		where the exponent of $t$ is the representative of the congruence class of $i-p$ in $\{0, ..., \ell-1\}$.

	\begin{defn}\label{defn:H}
		For $\blambda \in \scrP(\ell,n)$ we define $G_{\blambda} (x;t,t) \in \boldsymbol{\Lambda}(q,t)$ as the preimage of
		\begin{equation}
			 \prod_{j=0}^{\ell-1} G_{\lambda^{(j)}} (x;t^\ell,t^\ell) \left[Z^{(j)}\right] \label{eq:H}
		\end{equation}
		under $\textnormal{\textbf{ev}}_N$.
	\end{defn}

	\begin{defn} \label{defn:ttKostka}
		For $\blambda,\bmu\in\scrP(\ell,n)$ we define the \textit{$t,t$-Kostka--Macdonald coefficient} $K_{\bmu\blambda}(t,t)$ by
		\begin{equation}
			G_{\blambda} (x;t,t) = \sum_{\bmu \in \scrP(\ell,n)} K_{\bmu\blambda}(t,t) \cdot s_{\bmu}\;. \label{eq:ttKostka}
		\end{equation}
	\end{defn}

	\subsection{Generalization of Gordon's character formula}

	Let $W$ be the complex reflection group of type $G(\ell,1,n)$ and $\rHbc$ the restricted rational Cherednik algebra of $W$ in generic parameter. For a graded $\rHbc$-module $V$ we denote by $[V]$ its representative in the graded Grothendieck group of $W$.

	\begin{thm}\label{thm:ttKostka}
		For $\ell$-multipartitions $\blambda,\bmu \in \scrP(\ell,n)$ let $K_{\bmu\blambda}(t,t)$ be the $t,t$-Kostka--Macdonald coefficients as in Definition \ref{defn:ttKostka} and $[L(\blambda)]$ the simple $\rHbc$-module associated to $\lambda$. We then have that
		\begin{equation}
			[L(\blambda)] = \sum_{\bmu \in \scrP(\ell,n)} K_{\bmu\blambda}(t,t) \cdot [\bmu]\;.
		\end{equation}
	\end{thm}

	The proof will occupy the remainder of this subsection. To generalize the $\ell=1$ proof in \cite[Thm. 6.4$(ii)$]{gordon2003baby}, we have to show two properties of the wreath Kostka--Macdonald coefficients.

	We write $\mathbf{1}$ (resp. $1$) for the $\ell$-multipartition corresponding to the trivial representation, i.e. $\mathbf{1} = ((n), \emptyset, ..., \emptyset)$ (resp. $1 =(n)$). Let $\zeta$ be a primitive $\ell\thh$ root of unity.

	\begin{lem} \label{lem:z}
		Let $Z^{(p)}$ be the variables defined in \eqref{eq:Z}. We define two vectors
		\begin{equation}
			\mathbf{z} = \begin{pmatrix}
				Z^{(0)} & \cdots & Z^{(\ell-1)}
			\end{pmatrix}^{\tr}\;, \ \
			\mathbf{x} = \begin{pmatrix}
				X^{(0)} & \cdots & X^{(\ell-1)}
			\end{pmatrix}^{\tr}\;,
		\end{equation}
		where $(\cdot)^{\tr}$ denotes the transposition, and two matrices
		\begin{equation}
			T=(\zeta^{i\cdot j})_{0 \leq i,j \leq \ell-1}\;, \ \ D=\diag\left(\frac{1-t^\ell}{1-\zeta^0 t}\;, \ ... \ , \ \frac{1-t^\ell}{1-\zeta^{\ell-1}t}\right)\;.
		\end{equation}
		We then have
		\begin{equation}
			\overline{T} \cdot \mathbf{z} = D \cdot \overline{T} \cdot \mathbf{x} \label{eq:Tz=DTz}
		\end{equation}
		where $\overline{T}$ is obtained by complex conjugation of the entries of $T$.
	\end{lem}

	\begin{proof}
		We show this by solving the system of linear equations for $\mathbf{z}$. We first have to see that
			\begin{equation}
				T\cdot \overline{T} = \diag(\ell, ..., \ell)
			\end{equation}
		by row orthogonality of the character table of $C_\ell$ (which is equal to $T$). This means we have
		\begin{equation}
			T \cdot \overline{T} \cdot \mathbf{z} = \ell \cdot \mathbf{z}\;. \label{eq:ellz}
		\end{equation}
		On the right hand side of equation \eqref{eq:Tz=DTz}, we obtain after multiplication with $T$ for the $p\thh$ entry
			\begin{equation}
				(T\cdot D \cdot \overline{T} \cdot \mathbf{x} )_p = \sum_{j=0}^{\ell-1} \zeta^{p\cdot j} \left( \frac{1-t^\ell}{1-\zeta^j t} \sum_{i=0}^{\ell-1} \overline{\zeta^{i\cdot j}} X^{(i)} \right)\;. \label{eq:TDTxp}
			\end{equation}
		Combining \eqref{eq:ellz} and \eqref{eq:TDTxp}, we get for $Z^{(p)}$
		\begin{align*}
			Z^{(p)} &= \frac{1}{\ell} \cdot \sum_{j=0}^{\ell-1} \zeta^{p\cdot j} \left( \frac{1-t^\ell}{1-\zeta^j t} \sum_{i=0}^{\ell-1} \overline{\zeta^{i\cdot j}} X^{(i)} \right)\\
			&= \frac{1}{\ell} (1-t^\ell) \cdot \sum_{j=0}^{\ell-1} \zeta^{p\cdot j} \left( \frac{1}{1-\zeta^j t} \sum_{i=0}^{\ell-1} \overline{\zeta^{i\cdot j}} X^{(i)} \right)\\
			&= \frac{1}{\ell} (1-t^\ell) \cdot \sum_{i=0}^{\ell-1} \left ( \sum_{j=0}^{\ell-1} \frac{1}{1-\zeta^j t} \zeta^{p\cdot j - i\cdot j} \right ) X^{(i)}\\
			&= \frac{1}{\ell} (1-t^\ell) \cdot \sum_{i=0}^{\ell-1} \left ( \sum_{j=0}^{\ell-1} {\left(\zeta^{-(i-p)}\right)}^j \sum_{k=0}^{\infty} {(\zeta^j t)}^k \right) X^{(i)}\\
			&= \frac{1}{\ell} (1-t^\ell) \cdot \sum_{i=0}^{\ell-1} \left ( \sum_{k=0}^{\infty} t^k \sum_{j=0}^{\ell-1} {\left(\zeta^{k-(i-p)}\right)}^j \right) X^{(i)}\\
			\intertext{The inner sum over $j$ is equal to $\ell$ if $k-(i-p) \equiv 0$ mod $\ell$, and else $0$. Therefore, only every $\ell\thh$ $k$-summand remains. Note that $i-p$ is in $\{-(\ell-1), ..., \ell-1\}$. If $i-p$ is nonnegative, because $k\geq 0$, all $k$ of the form $k=\ell\cdot k' + (i-p)$ for $k' \geq 0$ remain. If $i-p$ is negative, all $k$ of the form $k = \ell\cdot k' + \ell + (i-p)$ for $k'\geq 0$ remain. That is equivalent to saying all $k$ of the form $k=\ell\cdot k' + \overline{(i-p)}$ for some $k'\geq 0$ remain where $\overline{(i-p)}$ denotes the unique element in $\{0, ..., \ell-1\}$ that is equivalent to $i-p$ mod $\ell$.}
			&= \frac{1}{\ell} (1-t^\ell) \cdot \sum_{i=0}^{\ell-1} \left ( \sum_{k'=0}^{\infty} \ell \cdot t^{\ell k' + \overline{(i-p)}} \right) X^{(i)} \\
			&= \frac{1}{\ell} (1-t^\ell) \cdot \sum_{i=0}^{\ell-1} \frac{\ell\cdot t^{\overline{(i-p)}}}{1-t^\ell} X^{(i)} \\
			&= \frac{1}{\ell} \cdot \sum_{i=0}^{\ell-1} \ell \cdot  t^{\overline{(i-p)}} \cdot  X^{(i)} \\
			&= \sum_{i=0}^{\ell-1}  t^{\overline{(i-p)}}  X^{(i)}
		\end{align*}
		That means that $Z^{(p)}$ is a cyclic permutation of $Z^{(0)} = \sum_{i=0}^{\ell-1} t^i X^{(i)}$ which agrees with their definition in \eqref{eq:Z}
	\end{proof}

	The first statement in the following proposition is an $\ell$-analogue of the traditional case (see the proof of \cite[IV.8.16]{macdonald2015symmetric})

	\begin{prop}\label{prop:12}
		Let $\brho=(\rho^{(0)}, ..., \rho^{(\ell-1)})\in \scrP(\ell,n)$ and recall that $H_\lambda(t)$ denotes the hook polynomial of a partition $\lambda$ defined in \eqref{eq:hook}. We then have the following properties of the $t,t$-Kostka--Macdonald coefficients.
		\begin{enumerate}[label=(\roman*)]
			\item $\sum_{\bmu}  K_{\bmu\blambda}(t,t) \ \chi^{\bmu}_{\brho} = \prod_{j=0}^{\ell-1} \left ( H_{\lambda^{(j)}}(t^\ell) \prod_i \left(1-\zeta^j t^{\rho^{(j)}_i}\right)^{-1} \right ) \chi^{\blambda}_{\brho}$\;,\\
			\item  $ K_{\bmu\mathbf{1}}(t,t) = f_{\bmu}(t)$\;.
		\end{enumerate}
	\end{prop}
	\begin{proof}
	$(i):$ We look at the preimages of the $G_{\blambda}(x;t,t)$ and $\sum_{\bmu} K_{\bmu\blambda}(t,t) s_{\bmu}$ under the Frobenius character map $\ch$ (see Definition \ref{defn:poi}) and use Lemma \ref{lem:z}.
	\begin{align*}
		G_{\blambda}(x;t,t) &= \prod_{j=0}^{\ell-1} G_{\lambda^{(j)}}(x;t^\ell, t^\ell) \ple \left[Z^{(j)}\right] = \prod_{j=0}^{\ell-1} H_{\lambda^{(j)}}(t^\ell) s_{\lambda^{(j)}} \ple \left[\frac{Z^{(j)}}{1-t^\ell} \right]\\
		&\!\overset{\scriptstyle{\ch}}{=} \sum_{\bmu \in \scrP(\ell,n)} z_{\bmu}\inv \cdot  \left( \prod_{j=0}^{\ell-1} H_{\lambda^{(j)}}(t^\ell) \chi_{\bmu}^{\blambda} \right)  \\
		&\ \ \ \ \ \ \ \ \ \ \ \ \ \ \cdot \prod_{j=0}^{\ell-1} p_{\mu^{(j)}}\hspace{-0.3em} \left[ \frac{\sum_{i=0}^{\ell-1} \overline{\zeta^{i\cdot j}} Z^{(i)}}{1-t^\ell} \right] \\
		&\!\overset{\scriptstyle{\ref{lem:z}}}{=} \sum_{\bmu \in \scrP(\ell,n)} z_{\bmu}\inv \cdot  \left( \prod_{j=0}^{\ell-1} H_{\lambda^{(j)}}(t^\ell) \chi_{\bmu}^{\blambda} \right)  \\
		&\ \ \ \ \ \ \ \ \ \ \ \ \ \ \cdot \prod_{j=0}^{\ell-1} p_{\mu^{(j)}}\hspace{-0.3em} \left[ \frac{\frac{1-t^\ell}{1-\zeta^j t} \cdot \sum_{i=0}^{\ell-1} \overline{\zeta^{i\cdot j}} X^{(i)}}{1-t^\ell} \right] \\
		&= \sum_{\bmu \in \scrP(\ell,n)} z_{\bmu}\inv \cdot  \left( \prod_{j=0}^{\ell-1} H_{\lambda^{(j)}}(t^\ell) \prod_{i=1}^{l(\mu^{(j)})} \left(1-\zeta^j t^{\mu^{(j)}_i}\right)^{-1} \chi_{\bmu}^{\blambda} \right)\\
		&\ \ \ \ \ \ \ \ \ \ \ \ \ \ \cdot \prod_{j=0}^{\ell-1} p_{\mu^{(j)}}\hspace{-0.3em} \left[ \sum_{i=0}^{\ell-1} \overline{\zeta^{i\cdot j}} X^{(i)} \right],
	\end{align*}
	\begin{equation}
		\sum_{\bmu} K_{\bmu\blambda}(t,t) s_{\bmu} = \sum_{\bmu \in \scrP(\ell,n)} z_{\bmu}\inv \cdot   \left( \sum_{\bmu} K_{\bmu\blambda}(t,t) \chi_{\bmu}^{\blambda} \right ) \cdot \prod_{j=0}^{\ell-1} p_{\mu^{(j)}}\hspace{-0.3em} \left[ \sum_{i=0}^{\ell-1} \overline{\zeta^{i\cdot j}} X^{(i)} \right].
	\end{equation}
	The claim now follows from the bijectivity of $\ch$.\\

	$(ii):$ Using $(i)$, we now show that
	\begin{equation}
		G_{\mathbf{1}}(x;t,t) = \sum_{\boldsymbol{\mu}} f_{{\boldsymbol{\mu}}}(t) \cdot s_{\boldsymbol{\mu}} \label{eq:Hfs}
	\end{equation}
	where again $\mathbf{1}$ denotes $((n), \emptyset, ..., \emptyset) \in \scrP(\ell,n)$ which corresponds to the trivial representation. Using Definition \ref{defn:H}  and Lemma \ref{lem:z} for $\blambda = \mathbf{1}$ we get
	\begin{equation}
		\prod_{j=0}^{\ell-1} G_{\lambda^{(j)}} (x;t^\ell,t^\ell) \left[Z^{(j)}\right] = G_1(x,t^\ell,t^\ell)\left[Z^{(0)}\right]
	\end{equation}
	which transforms \eqref{eq:Hfs} into
	\begin{equation}
		G_1(x;t^\ell,t^\ell)\left[Z^{(0)}\right] = \sum_{\boldsymbol{\mu}} f_{{\boldsymbol{\mu}}}(t) \cdot s_{\boldsymbol{\mu}} \left[Z^{(0)}\right]
	\end{equation}
	which we solve by using the definition of the type $A$ Kostka--Macdonald coefficients. For $\brho\in\scrP(\ell,n)$ we use the notation $\alpha(\boldsymbol{\rho})$ for the vector $(|\rho^{(0)}|, ..., |\rho^{(0)}|)$ for which we will take the $b$-invariant as defined in \eqref{eq:binv}. Note that here we implicitly use $\mathbf{ev}_N$ to go between Schur functions and Schur polynomials.

	\begin{align*}
		G_{1}(x;t^m,t^m) \left [Z^{(0)}\right] &= \sum_{\mu} K_{\mu\, 1}(t^m,t^m) \cdot s_\mu \ple \left [Z^{(0)}\right] \\
		&\!\!\overset{\scriptstyle{\eqref{eq:Z}}}{=} \sum_{\mu} K_{\mu\,1}(t^m,t^m) \cdot s_\mu \ple \left [ \sum_{i=1}^{\ell-1}  t^i  X^{(i)} \right] \\
		&\!\overset{\scriptstyle{\ref{cor:sX+Y}}}{=} \sum_{\mu} K_{\mu\,1}(t^m,t^m) \cdot \sum_{\brho} c_{\boldsymbol{\rho}}^\mu \prod_{i=1}^{\ell-1} 	s_{\rho^{(i)}} \ple \left[ t^i X^{(i)} \right ] \\
		&= \sum_{\mu} K_{\mu\,1}(t^m,t^m) \cdot \sum_{\brho} c_{\brho}^\mu \ t^{b(\alpha(\boldsymbol{\rho}))} \cdot \prod_{i=1}^{\ell-1}  s_{\rho^{(i)}} \ple \left[ X^{(i)} \right ] \\
		&= \sum_{\mu} K_{\mu\,1}(t^m,t^m) \cdot \sum_{\boldsymbol{\rho}} c_{\boldsymbol{\rho}}^\mu \ t^{b(\alpha(\boldsymbol{\rho}))}  s_{\boldsymbol{\rho}} \\
		&= \sum_{\boldsymbol{\rho}} \left ( \sum_{\mu} K_{\mu\,1}(t^m,t^m) \ c_{\boldsymbol{\rho}}^\mu \ t^{b(\alpha(\boldsymbol{\rho}))} \right) \cdot s_{\boldsymbol{\rho}}
	\end{align*}

	We now calculate the coefficient of $s_{\boldsymbol{\rho}}$. We have $K_{\mu\,1}(t,t) = f_\mu(t)$ by \cite[VI.8 Ex. 1]{macdonald1988new} and Theorem \ref{thm:fakedeg}. Using this, we get
	\begin{align*}
		\sum_{\mu} &K_{\mu\,1}(t^\ell,t^\ell) \ c_{\boldsymbol{\rho}}^\mu \ t^{b(\alpha(\boldsymbol{\rho}))}\\
		&= t^{b(\alpha(\boldsymbol{\rho}))} \sum_{\mu} K_{\mu\,1}(t^\ell,t^\ell) \ c_{\boldsymbol{\rho}}^\mu \\
		&= t^{b(\alpha(\boldsymbol{\rho}))} \sum_{\mu} f_\mu (t^\ell) \ c_{\boldsymbol{\rho}}^\mu \\
		&= t^{b(\alpha(\boldsymbol{\rho}))} \cdot (1-t^\ell) (1-t^{2\cdot \ell}) \cdots (1-t^{n\cdot \ell}) \cdot \sum_{\mu} s_\mu \ple \left [\frac{1}{1-t^\ell} \right ]\cdot c_{\boldsymbol{\rho}}^\mu \\
		&= t^{b(\alpha(\boldsymbol{\rho}))} \cdot (1-t^\ell) (1-t^{2\cdot \ell}) \cdots (1-t^{n\cdot \ell}) \cdot \prod_{i=1}^{\ell-1} s_{\rho^{(i)}} \ple \left [\frac{1}{1-t^\ell} \right ] \\
		&= f_{\boldsymbol{\rho}}(t)\;.
	\end{align*}

	\end{proof}

	We now have all the tools to prove Theorem \ref{thm:ttKostka}. As in \cite[Thm. 6.4$(ii)$]{gordon2003baby}, we compare the corresponding characters of
	\begin{equation}
		[M(\blambda)]\ \ \text{and} \ \  t^{-b(\blambda)} f_{\blambda}(t) \cdot \sum_{\bmu} K_{\bmu  \blambda}(t,t) [\blambda]\;.
	\end{equation}
	 The second expression is obtained by first decomposing $M(\blambda)$ into $L(\blambda)$ using Proposition \ref{prop:|W|} and Remark \ref{rem:barf}, and then the $L(\blambda)$ into $\bmu$ by putting in the character formula we want to prove. The character of $[M(\blambda)]$ is determined by \eqref{eq:Mchar}.

	 All in all, we must show that for $\brho \in \scrP(\ell,n)$ the equation
	 \begin{equation}
	 	\sum_{\bmu} f_{\bmu}(t) \chi^{\bmu}_{\brho} \chi^{\blambda}_{\brho} = t^{-b(\blambda)} f_{{\blambda}}(t) \sum_{\bmu}  K_{\bmu  \blambda}(t,t) \chi^{\bmu}_{\brho} \label{eq:proofstart}
	 \end{equation}
	holds. We will only manipulate the right hand side of the equation. Start by using Proposition \ref{prop:12}$(i)$ to obtain
	\begin{equation}
		t^{-b({\blambda})} f_{\blambda}(t) \prod_{j=0}^{\ell-1} \left(H_{\lambda^{(j)}}(t^\ell) \prod_{i=0}^{\ell-1} \left(1-\zeta^j t^{\rho^{(j)}_i}\right)^{-1} \right) \chi^{{\blambda}}_{{\brho}}\;.
	\end{equation}
	We plug in the fake degree formula from Theorem \ref{thm:fakedeg} and get
	\begin{equation}
		\prod_{k=1}^n (1-t^{k\cdot \ell}) \prod_{j=0}^{\ell-1} \prod_i \left(1-\zeta^j t^{\rho^{(j)}_i}\right)^{-1} \chi^{{\blambda}}_{{\brho}}. \label{eq:prehook}
	\end{equation}
	We multiply with $\chi_{\brho}^{\mathbf{1}} = 1$, and write the first product as the hook polynomial of $\mathbf{1}$. Expression \eqref{eq:prehook} now reads
	\begin{equation}
	\prod_{j=0}^{\ell-1} \left( H_{\mathbf{1}^{(j)}}(t^\ell)  \prod_i \left(1-\zeta^j t^{\rho^{(j)}_i}\right)^{-1}\right) \chi_{\brho}^{\mathbf{1}} \cdot \chi^{{\blambda}}_{{\brho}}.
	\end{equation}
	We use Proposition \ref{prop:12}$(i)$ for $\blambda=\mathbf{1}$ and get
	\begin{equation}
		\sum_{{\bmu}}  K_{\bmu\mathbf{1}} \chi^{{\bmu}}_{{\brho}} \chi^{{\blambda}}_{{\brho}}\;.
	\end{equation}
	Equation \eqref{eq:proofstart} from the beginning now reduces to
	\begin{equation}
		\sum_{\bmu} f_{\bmu}(t) \cdot \chi^{\bmu}_{\brho} \chi^{\blambda}_{\brho} = \sum_{{\bmu}}  K_{\bmu\mathbf{1}} \cdot \chi^{{\bmu}}_{{\brho}} \chi^{{\blambda}}_{{\brho}}\;,
	\end{equation}
	which follows from Proposition \ref{prop:12}$(ii)$. This concludes the proof of Theorem \ref{thm:ttKostka}.

	\begin{rem}
		The first statement of Proposition \ref{prop:12} completely determines our Kostka--Macdonald coefficients (similar to the traditional case, see \cite[VI.8.17]{macdonald1988new}). Namely, if we take the scalar product with an irreducible character $\chi^{\bmu}$ on both sides of Proposition \ref{prop:12}$(i)$, we obtain
			\begin{equation}
				 K_{\bmu\blambda}(t,t) \  = \sum_{\brho \in \scrP(\ell,n) } z_{\brho}\inv \cdot \prod_{j=0}^{\ell-1} \left ( H_{\lambda^{(j)}}(t^\ell) \prod_{i=0}^{\ell-1} \left(1-\zeta^j t^{\rho^{(j)}_i}\right)^{-1} \right ) \cdot \chi^{\blambda} _{\brho} \  \chi^{\bmu}_{\brho}
			\end{equation}
		\end{rem}

	\subsection{The main theorem} \label{sec:main}

	We show that the characters we constructed are equal to a specialized version of the wreath Macdonald polynomials proposed by Haiman in \cite[Conjecture 7.2.19]{haiman2003combinatorics}. These polynomials are proven to exist in \cite{bezrukavnikov2014wreath}. Here, we index the wreath Macdonald polynomial by their $\ell$-quotient. Note that in general the wreath Macdonald polynomials depend on the choice of an $\ell$-core, but the specialization we work with does not.

	\begin{thm}\label{cor:Haiman}
		For $\blambda \in \scrP(\ell,n)$ let $H_{\blambda}(x;q,t)$ be a wreath Macdonald polynomial of $\blambda$ for some $\ell$-core. We have
		\begin{equation}
			\ch L(\blambda) =  t^{b(\blambda^*)} H_{\blambda}(x;t,t\inv)
		\end{equation}
		where $\ch L(\blambda)$ denotes image of the $W$-character of $L(\blambda)$ under the Frobenius character map.
	\end{thm}

	\begin{proof}
		We use the characterization found in \cite[2.3.6]{wen2019wreath} and the statement of Theorem \ref{thm:ttKostka}. When we specialize $H_{\blambda}(x;q,t)$ by $q\mapsto t$ and $t\mapsto t\inv$, the plethysms $\Phi\inv_q$ and $\Phi\inv_{t\inv}$ in \cite[Lemma 2.1]{wen2019wreath} agree and the two conditions \cite[Prop. 2.4(i)+(ii)]{wen2019wreath} reduce to
		\begin{equation}
			H_{\blambda}(x;t,t\inv) \in \bbC(t) \cdot \prod_{j=0}^{\ell-1} s_{\lambda^{(j)}} \left[ \frac{Z^{(j)}}{1-t^\ell} \right] \label{eq:calH<->s}
		\end{equation}
		where the $Z^{(j)}$ are the ones defined in \eqref{eq:Z}. Now, by the third characteristic of wreath Macdonald polynomials in \cite[2.3.6]{wen2019wreath} we have
		\begin{equation}
			\left< H_{\blambda}(x;t,t\inv), \mathbf{1} \right> = 1\;.
		\end{equation}
		Now, the trivial representation appears in $M(\blambda)$ only in the degrees in which the coinvariant space contains the dual representation $\blambda^*$. These degrees are given by $f_{\blambda^*}(t)$. Using \cite[Thm. 4.3]{stembridge1989eigenvalues}, we see that dualizing $\blambda=(\lambda^{(0)}, \lambda^{(1)}, \lambda^{(2)}, ..., \lambda^{(\ell-1)})$ yields the $\ell$-multipartition $\blambda^*=(\lambda^{(0)}, \lambda^{(\ell-1)}, \lambda^{(\ell-2)} ..., \lambda^{(1)})$. This means we have $\bar f_{\blambda}(t) = \bar f_{\blambda^*}(t)$ by Theorem \ref{thm:fakedeg}. Therefore by Proposition \ref{prop:|W|}, the trivial representation appears in $L(\blambda)$ only once, namely in degree $b(\blambda^*)$. Using now \eqref{eq:calH<->s} we have determined $H_{\blambda}(x;t,t\inv)$ uniquely and proven the claim.
	\end{proof}

	\subsection{Type $B_2$ example} \label{sec:B}

	We will give $B_2 = G(2,1,2) = C_2 \wr \frS_2$ as an example. Denote the variables in $\boldsymbol{\Lambda}_N$ by $X=X^{(0)}$ and $Y=X^{(1)}$. The irreducible representations of $B_2$ are parameterized by the bipartitions

	\begin{equation}
		\ytableausetup{boxsize=1em}
		\begin{array}{cccccc}
			\begin{ytableau}
			\empty & \empty \\
 		\end{ytableau},\emptyset &
 		\begin{ytableau}
			\empty \\
			\empty \\
 		\end{ytableau},\emptyset &
 		\begin{ytableau}
			\empty \\
 		\end{ytableau},
 		\begin{ytableau}
			\empty \\
 		\end{ytableau} & \emptyset,
 		\begin{ytableau}
			\empty \\
			\empty \\
 		\end{ytableau} & \emptyset,
 		\begin{ytableau}
			\empty & \empty \\
 		\end{ytableau}\\
 		\\
		((2),\emptyset) & ((1,1),\emptyset) & ((1),(1)) & (\emptyset,(1,1)) & (\emptyset,(2))
		\end{array}
	\end{equation}
	with grade-shifted fake degrees and hook polynomials
	\begin{equation}\label{eq:B2barf}
		\bar f_{\blambda}(t) =\begin{cases}
			1 + t^2 &\text{if} \  \blambda = ((1),(1))\;, \\
			1 & \text{otherwise}
		\end{cases}
	\end{equation}
	\begin{equation}
		H_{\lambda^{(0)}}(t^2) \cdot H_{\lambda^{(1)}}(t^2) =\begin{cases}
			{(1-t^2)}^2 &\text{if} \  \blambda = ((1),(1))\;, \\
			(1-t^2)(1-t^4) & \text{otherwise}
		\end{cases}
	\end{equation}
	The character table of $C_2=\left<g \mid g^2=1\right>$ is equal to
	\begin{equation} \label{eq:C2}
		\begin{array}{c|cc}
			 & 1 & g \\
			 \hline
			\chi_1 & 1 & 1\\
			\chi_2 & 1 & -1
		\end{array}
	\end{equation}
	\noindent which means the polynomial Frobenius character map is defined as

	\begin{equation}
		\chi \mapsto \sum_{\bmu \in \scrP(2,n)} z_{\bmu}\inv \cdot \chi_{\bmu} \cdot   p_{\mu^{(0)}}\hspace{-0.3em} \left[ X+Y \right] \cdot p_{\mu^{(1)}}\hspace{-0.3em} \left[ X - Y \right]\;.
	\end{equation}

	Since the character table of $\frS_2$ is also given by \eqref{eq:C2}, one can work out the Schur polynomials and $G_{\blambda}(x;t,t)$ by hand, the latter of which is given by the formula
	\begin{equation}
		H_{\lambda^{(0)}}(t^2)\, s_{\lambda^{(0)}} \left[\frac{X+tY}{1-t^2}\right] \cdot  H_{\lambda^{(1)}}(t^2)\, s_{\lambda^{(1)}} \left[\frac{tX+Y}{1-t^2}\right]\;.
	\end{equation}

	From the grade-shifted fake degrees \eqref{eq:B2barf} we obtain the decomposition matrix $D_\Delta$. The other two are given by

	\begin{equation}
		{
			\begin{array}{l||r|r|r|r|r|r|}
				 & ( ( 2 ), \emptyset ) & ( ( 1, 1 ), \emptyset ) & ( ( 1 ), ( 1 ) ) & ( \emptyset, ( 2 ) ) & ( \emptyset, ( 1, 1 ) )\\
			\hhline{=::=:=:=:=:=:}
			M(( ( 2 ), \emptyset )) & 1 & t^2 & t^3+t & t^2 & t^4  \\
			\hhline{-||-|-|-|-|-||-|}
			M(( ( 1, 1 ), \emptyset )) & t^2 & 1 & t^3+t & t^4 & t^2 \\
			\hhline{-||-|-|-|-|-||-|}
			M(( ( 1 ), ( 1 ) )) & t^3+t & t^3+t & t^4+2t^2+1 & t^3+t & t^3+t \\
			\hhline{-||-|-|-|-|-||-|}
			M(( \emptyset, ( 2 ) )) & t^2 & t^4 & t^3+t & 1 & t^2 \\
			\hhline{-||-|-|-|-|-||-|}
			M(( \emptyset, ( 1, 1 ) )) & t^4 & t^2 & t^3+t & t^2 & 1 \\
			\hhline{-||-|-|-|-|-|}
			\end{array}}
	\end{equation}\

	\noindent for $C_\Delta$, and

	\begin{equation}\label{eq:B2CL}
		{
			\begin{array}{l||r|r|r|r|r|r|}
				 & ( ( 2 ), \emptyset ) & ( ( 1, 1 ), \emptyset ) & ( ( 1 ), ( 1 ) ) & ( \emptyset, ( 2 ) ) & ( \emptyset, ( 1, 1 ) )\\
			\hhline{=::=:=:=:=:=:}

			L(( ( 2 ), \emptyset )) & 1 & t^2 & t^3+t & t^2 & t^4 \\
			\hhline{-||-|-|-|-|-|}
			L( ( 1, 1 ), \emptyset )) & t^2 & 1 & t^3+t & t^4 & t^2 \\
			\hhline{-||-|-|-|-|-|}
			L(( ( 1 ), ( 1 ) )) & t & t & t^2+1 & t & t \\
			\hhline{-||-|-|-|-|-|}
			L(( \emptyset, ( 2 ) )) & t^2 & t^4 & t^3+t & 1 & t^2 \\
			\hhline{-||-|-|-|-|-|}
			L(( \emptyset, ( 1, 1 ) ) & t^4 & t^2 & t^3+t & t^2 & 1 \\
			\hhline{-||-|-|-|-|-|}
			\end{array}}
	\end{equation}\

	\noindent for $C_L = (K_{\bmu\blambda}(t,t))_{\blambda,\bmu}	$. \\

	For equal Etingof--Ginzburg parameter $\bc=(1,1)$ we have that $C_L$ specializes to the matrix below.
	\begin{equation*}
		\begin{array}{c||c|c|c|c|c|}
		 & ( ( 2 ), \emptyset ) & ( ( 1, 1 ), \emptyset ) & ( ( 1 ), ( 1 ) ) & ( \emptyset, ( 2 ) ) & ( \emptyset, ( 1, 1 ) )\\
			\hhline{=::=|=|=|=|=|}
			L(( ( 2 ), \emptyset )) & 1 & t^2 & t+t^3 & t^2 & t^4 \\
			\hhline{-||-|-|-|-|-|}
			L( ( 1, 1 ), \emptyset )) & 0 & 1 & 0 & 0 & 0 \\
			\hhline{-||-|-|-|-|-|}
			L(( ( 1 ), ( 1 ) )) & t & 0 & 1+t^2 & 0 & t \\
			\hhline{-||-|-|-|-|-|}
			L(( \emptyset, ( 2 ) )) & 0 & 0 & 0 & 1 & 0 \\
			\hhline{-||-|-|-|-|-|}
			L(( \emptyset, ( 1, 1 ) ) & t^4 & t^2 & t+t^3 & t^2 & 1 \\
			\hhline{-||-|-|-|-|-|}
		\end{array}
	\end{equation*}

	\printbibliography

\end{document}